\documentclass[11pt]{article}
\usepackage{amsmath,amssymb,amsthm,graphicx,float,url,color}
\usepackage[colorlinks=true,citecolor=red,filecolor=red,urlcolor=magenta]{hyperref}
\usepackage{pdfsync}
\usepackage{xcolor}
\usepackage{enumitem}
\usepackage{dsfont}

\newcommand{\N} {\mathbb{N}}
\newcommand{\Z} {\mathbb{Z}}

\newcommand{\R} {\mathbb{R}}
\newcommand{\C} {\mathbb{C}}
\newcommand{\g}{\mathfrak{g}}
\newcommand{\QL}{\mathcal{Q}}
\newcommand{\Op}{\mathrm{Op}}
\newcommand{\supp}{\mathrm{supp}}

\renewcommand{\geq}{\geqslant}
\renewcommand{\leq}{\leqslant}

\newtheorem{theorem}{Theorem}  
\newtheorem{proposition}{Proposition}

\newtheorem{definition}{Definition}
\newtheorem{lemma}{Lemma}

\theoremstyle{definition}\newtheorem{remark}{Remark}

\title{Quantum Limits on product manifolds}

\author{Emmanuel Humbert\footnote{Institut Denis Poisson, UFR Sciences et Technologie, Facult\'e Fran\c{c}ois Rabelais, Parc de Grandmont, 37200 Tours, France (\texttt{emmanuel.humbert@lmpt.univ-tours.fr}).}
\and Yannick Privat\footnote{IRMA, Universit\'e de Strasbourg, CNRS UMR 7501, 7 rue Ren\'e Descartes, 67084 Strasbourg, France (\texttt{yannick.privat@unistra.fr}).}
	\and Emmanuel Tr\'elat\footnote{Sorbonne Universit\'e, CNRS, Universit\'e de Paris, Inria, Laboratoire Jacques-Louis Lions (LJLL), F-75005 Paris, France (\texttt{emmanuel.trelat@sorbonne-universite.fr}).}}

\date{}

\begin{document}

\maketitle

\begin{abstract}
We establish some properties of quantum limits on a product manifold, proving for instance that, under appropriate assumptions, the quantum limits on the product of manifolds are absolutely continuous if the quantum limits on each manifolds are absolutely continuous. On a product of Riemannian manifolds satisfying the minimal multiplicity property, we prove that a periodic geodesic can never be charged by a quantum limit.
\end{abstract}

\section{Introduction and main results}
\subsection{Setting}\label{sec_setting}
We first recall the definition of a \textit{Quantum Limit} (QL in short). Let $M$ be a smooth compact manifold, endowed with a probability measure $\rho$, and let $\triangle$ be a self-adjoint operator on $L^2(M,\rho)$ (the space of square $\rho$-integrable complex-valued functions on $M$), bounded below and having a compact resolvent and thus a discrete spectrum $\mathrm{Spec}(\triangle)$. 

A \emph{local QL} (or \emph{QL on the base}) of $\triangle$ is a probability Radon measure on $M$ that is a closure point (weak limit) of the family of probability measures $\mu_{\phi_\lambda} = \vert\phi_\lambda\vert^2\, \rho$ (for $\lambda\in\mathrm{Spec}(\triangle)$) as $\lambda\rightarrow+\infty$, where $\triangle\phi_\lambda=\lambda\phi_\lambda$ and $\Vert\phi_\lambda\Vert_{L^2(M,\rho)}=1$. The set $\QL(M)$ of local QLs of $\triangle$ is compact. Actually, to define local QLs, it suffices that $M$ be a topological compact set.

Given any quantization $\Op$, a \emph{microlocal QL} of $\triangle$ is a probability Radon measure on the co-sphere bundle $S^*M$ that is a closure point of the family of Radon measures $\tilde\mu_{\phi_\lambda}$ (for $\lambda\in\mathrm{Spec}(\triangle)$) as $\lambda\rightarrow+\infty$, where $\tilde\mu_{\phi_\lambda}(a) = \langle\Op(a)\phi_\lambda,\phi_\lambda\rangle$ for every classical symbol $a$ of order $0$ and $\phi_\lambda$ is an eigenfunction of $\triangle$ of norm $1$ corresponding to the eigenvalue $\lambda$.
The set $\QL(S^*M)$ of microlocal QLs of $\triangle$ is compact. 

Every local QL is the image of a (not necessarily unique) microlocal QL under the canonical projection $\pi:S^*M\rightarrow M$, i.e., we have $\pi_*\QL(S^*M)=\QL(M)$.
Indeed, given any $f\in C^0(M)$ and any eigenfunction $\phi$ of $\triangle$, we have
$$
(\pi_*\tilde\mu_\phi)(f) = \tilde\mu_\phi(\pi^*f) = \langle\Op(\pi^*f)\phi,\phi\rangle = \int_M f \vert\phi\vert^2\, d\rho 
$$
because $\Op(\pi^*f)\phi=f\phi$. The equality then easily follows by weak compactness of probability Radon measures.

Here and throughout, the co-sphere bundle $S^*M$ is defined as the sphere bundle $ST^*M$ of the cotangent bundle, that is, the quotient of $T^*M\setminus\{0\}$ under positive homotheties. Homogeneous functions of order $0$ on $T^*M\setminus\{0\}$ are identified with functions on $S^*M$.

\medskip

Our objective is to give results for QLs of \emph{product} manifolds.

For $\ell=1,2$, let $M_\ell$ be a smooth compact manifold endowed with a Borel probability measure $\rho_\ell$, and let $\triangle_\ell$ be a self-adjoint operator on $L^2(M_\ell,\rho_\ell)$, bounded below and having a compact resolvent and thus a discrete spectrum $\lambda^\ell_1 < \lambda^\ell_2 < \cdots < \lambda^\ell_i < \cdots$ 
where each $\lambda^\ell_i$ is of multiplicity $m^\ell_i\in\N^*$ and $\lambda^\ell_i\rightarrow+\infty$ as $i\rightarrow+\infty$. 
We consider the product manifold $M=M_1\times M_2$, endowed with the product probability measure $\rho=\rho_1\otimes \rho_2$ and the self-adjoint operator $\triangle$ on $L^2(M,\rho)$ defined by
$$
\triangle = (\triangle_1)_{x_1}+(\triangle_2)_{x_2} = \triangle_1\otimes\mathrm{id}_{M_2} + \mathrm{id}_{M_1}\otimes\triangle_2 .
$$
The eigenvalues of $\triangle$ (not ordered in a nondecreasing way) are given by $\lambda_{i,j} = \lambda^1_i+\lambda^2_j$, for $i,j\in\N^*$. Given any pair $(i,j)\in(\N^*)^2$, we denote by $m_{i,j}$ the multiplicity of the eigenvalue $\lambda_{i,j}$. Since it may happen that $\lambda_{i,j} = \lambda_{i',j'} = \lambda^1_{i'}+\lambda^2_{j'}$ for another pair $(i',j')\in(\N^*)^2$, we always have
$$
m_{i,j} \geq m^1_i m^2_j .
$$
For $\ell=1,2$, given any Hilbert basis $(\phi^\ell_{i,k})_{i\in\N^*,\, k\in\{1,\ldots,m^\ell_i\}}$ of $L^2(M_\ell,\rho_\ell)$ consisting of eigenfunctions of $\triangle_\ell$, i.e.,
\begin{equation}\label{Hilbertbasis}
\triangle_\ell \phi^\ell_{i,k} = \lambda^\ell_i \phi^\ell_{i,k} \qquad \forall i\in\N^* \qquad \forall k\in\{1,\ldots,m^\ell_i\},
\end{equation}
a Hilbert basis (among others) of $L^2(M,\rho)$ of eigenfunctions of $\triangle$ is given by the functions $\phi^1_{i,k} \otimes \phi^2_{j,l}$, for $i,j\in\N^*$, $k\in\{1,\ldots,m^1_i\}$ and $l\in\{1,\ldots,m^2_j\}$.

Obviously, we have the inclusion
$$
\vert\phi^1_{i,k}\vert^2\, \rho_1\otimes\QL(M_2) \ \cup\ \QL(M_1)\otimes\vert\phi^2_{j,l}\vert^2\, \rho_2 \ \cup\ \QL(M_1)\otimes\QL(M_2) \ \subset\ \QL(M) 
$$
for all $i,j\in\N^*$, $k\in\{1,\ldots,m^1_i\}$ and $l\in\{1,\ldots,m^2_j\}$, showing that there is an infinite number of local QLs. However, QLs may have a much more complicated structure.

Note that, for any $\mu_2\in\QL(M_2)$ and any $\tilde\mu_2\in\QL(S^*M_2)$ such that $(\pi_2)_*\tilde\mu_2=\mu_2$ where $\pi_2:S^*M_2\rightarrow M_2$ is the canonical projection, for any fixed $i\in\N^*$ and $k\in\{1,\ldots,m^1_i\}$, the QL $\vert\phi^1_{i,k}\vert^2\, \rho_1\otimes\mu_2$ is the image under $\pi$ of the microlocal QL $\vert\phi^1_{i,k}\vert^2\, \rho_1\,\delta_{\xi_1=0}\otimes\tilde\mu_2$ (and symmetrically, by switching the roles of $M_1$ and $M_2$).

\subsection{Main results}\label{sec_mainresults}

We consider the following assumptions, which will be commented hereafter:
\begin{enumerate}[label=$\bf (UC)$,leftmargin=1.33cm]
\item\label{UC} (Unique continuation property) For $\ell=1,2$, any eigenfunction of $\triangle_\ell$ vanishing on a measurable subset of $M_\ell$ of positive measure is identically zero.
\end{enumerate}
\begin{enumerate}[label=$\bf (B)$,leftmargin=1.33cm]
\item\label{B} (Boundedness property) For $\ell=1,2$, any eigenfunction of $\triangle_\ell$ is bounded (not necessarily uniformly).
\end{enumerate}
\begin{enumerate}[label=$\bf (MM)$,leftmargin=1.33cm,topsep=0cm]
\item\label{MM} (Minimal multiplicity property) $m_{i,j} = m^1_i m^2_j$ for all $i,j\in\N^*$.
\end{enumerate}

\begin{theorem}\label{thm1}
\begin{enumerate}[label=(\roman*)]
\item\label{itemi} Under \ref{UC} and \ref{MM}, 
if every local QL of $M_\ell$ has full support, for $\ell=1,2$, then every local QL of $M$ has full support.\footnote{Equivalently, by compactness of QLs, $\displaystyle\inf_{\mu \in \QL(M) } \mu(\omega)>0$ for any open subset $\omega \subset M$.}
\item\label{itemii} Under \ref{B} and \ref{MM}, if every local QL of $M_\ell$ 
is absolutely continuous with respect to $\rho_\ell$, for $\ell=1,2$, then every local QL of $M$ 
is absolutely continuous with respect to $\rho$.
\end{enumerate}
\end{theorem}

Note that Theorem \ref{thm1} is valid more generally if $M_\ell$ is a compact metric set endowed with a Borel probability measure $\rho_\ell$, for $\ell=1,2$.

\medskip

For the next result, we assume moreover (in addition to the definitions of Section \ref{sec_setting}) that $\triangle_1$ and $\triangle_2$ are differential operators having the same order $d\in\N^*$ (which must be even), of respective principal symbols $h_1$ and $h_2$. 
With a slight abuse of notation, we consider the functions $h_1$ and $h_2$ as smooth functions on the cotangent bundle $T^*M$, homogeneous of order $d$ (rigorously, we should speak of $h_1\otimes 1$ and of $1\otimes h_2$). 
Then, the smooth function $h=h_1+h_2$ is the principal symbol (of order $d$) of the differential operator $\triangle$. 
Considering if necessary $\triangle_\ell+c_\ell\,\mathrm{id}$ for some $c_\ell\geq 0$ large enough, without loss of generality, we assume that $\triangle_1$ and $\triangle_2$ are nonnegative, so that $h_1\geq 0$ and $h_2\geq 0$.

Note that,\footnote{Indeed, $h_\ell(x,\lambda\xi)=\lambda^d h_\ell(x,\xi)$ for every $\lambda>0$ implies $\frac{\partial h_\ell}{\partial\xi}(x,\xi).\xi=d\times h_\ell(x,\xi)$.}
since $h_\ell$ is homogeneous of degree $d$, denoting by $\pi_\ell:T^*M_\ell\rightarrow M_\ell$ the canonical projection, and by $\vec{h}_\ell$ the Hamiltonian vector field corresponding to the Hamiltonian $h_\ell$, we have $d\pi_\ell.\vec{h}_\ell\neq 0$, i.e., $\frac{\partial h_\ell}{\partial\xi_\ell}\neq 0$ in local coordinates, on the open subset $\{h_\ell>0\}$ of $T^*M$, for $\ell=1,2$. 
In particular, the set $\{h_\ell=1\}$ is a submanifold of $T^*M$.

We have three (nonnegative) Hamiltonians on $T^*M$ and thus three natural Hamiltonian flows on $T^*M$: 
\begin{itemize}
\item The standard Hamiltonian flow, corresponding to the Hamiltonian $h$. A Hamiltonian curve (of $h$) on $T^*M$ is, by definition, an absolutely continuous curve solution of
$$
\dot x_1 = \frac{\partial h_1}{\partial\xi_1}(x_1,\xi_1), \quad \dot x_2 = \frac{\partial h_2}{\partial\xi_2}(x_2,\xi_2), \quad \dot\xi_1 = -\frac{\partial h_1}{\partial x_1}(x_1,\xi_1), \quad \dot\xi_2 = -\frac{\partial h_2}{\partial x_2}(x_2,\xi_2),  
$$
in local coordinates, i.e., an integral curve of the Hamiltonian vector field $\vec{h}=\vec{h}_1+\vec{h}_2$.
%
\item The {\em horizontal Hamiltonian flow}, corresponding to the Hamiltonian $h_1$ viewed as a smooth function on $T^*M$.
A horizontal Hamiltonian curve on $T^*M$ is, by definition, an absolutely continuous curve solution of
$$
\dot x_1 = \frac{\partial h_1}{\partial\xi_1}(x_1,\xi_1), \quad \dot x_2 = 0, \quad \dot\xi_1 = -\frac{\partial h_1}{\partial x_1}(x_1,\xi_1), \quad \dot\xi_2 = 0,  
$$
in local coordinates, i.e., an integral curve of the Hamiltonian vector field $\vec{h}_1$.
Note that the horizontal Hamiltonian flow differs from the standard Hamiltonian flow, although they coincide when both are restricted to $h_2=0$. 
\item The {\em vertical Hamiltonian flow}, defined similarly to the horizontal Hamiltonian flow but exchanging the roles of $M_1$ and $M_2$.
\end{itemize}
Note that $h_1$ and $h_2$ are commuting first integrals of the three Hamiltonian flows.


\medskip

A typical example (that we hereafter refer to as ``Riemannian case") is when $M_\ell$ is a smooth compact connected Riemannian manifold (without boundary) with a smooth Riemannian metric, $\rho_\ell$ is the normalized canonical measure, and $\triangle_\ell$ is the Laplace-Beltrami operator on $M_\ell$, for $\ell=1,2$.\footnote{Note that \ref{UC} and \ref{B} are then satisfied.}
In this case, the Hamiltonian flow is the usual geodesic flow, we consider also what we call the horizontal and the vertical geodesic flows, and we have $h_\ell=g^*_\ell$ where $g^*_\ell$ is the co-metric on $T^*M_\ell$.
This is an elliptic case, but subelliptic H\"ormander operators, like sub-Riemannian Laplacians, can also be considered; then their characteristic manifold is nontrivial.

\medskip

We define the three (conic) sets $\Sigma_1$, $\Sigma_2$ and $\Sigma$ of $T^*M$ by $\Sigma_1=\{h_1=0\}$, $\Sigma_2=\{h_2=0\}$ and $\Sigma=\{h=0\}=\Sigma_1\cap\Sigma_2$. For $\ell=1,2$, the canonical projection of $\Sigma_\ell$ onto $T^*M_\ell$ is the characteristic variety of $\triangle_\ell$. Besides, $\Sigma$ is the characteristic variety of $\triangle$. 
We define $S\Sigma$ as the quotient of $\Sigma$ under the action of positive homotheties acting on $\xi$ in local coordinates $(x,\xi)$. As in \cite{CHT_Duke}, the co-sphere bundle $S^*M$ is identified with the compactification of $\{h=1\}$, which is (at any point of $M$) the set $\{h=1\}$ glued at infinity with $S\Sigma$. The submanifold $\{h=1\}$ of $T^*M$ is identified with an open subset of $S^*M$.



\begin{theorem}\label{thm2}
%
Let $\tilde\mu\in\QL(S^*M)$ be a microlocal QL and let $\mu=\pi_*\tilde\mu\in\QL(M)$. 
The measure $\tilde\mu$ is invariant under the Hamiltonian flow.
Under \ref{MM}:
\begin{enumerate}[label=(\roman*),start=3]
\item\label{itemiii} $\tilde\mu$ is also invariant under both the horizontal and vertical Hamiltonian flows.
\item\label{itemiv} Given any periodic Hamiltonian curve $\tilde\gamma$ on $\{h=1\}$, 
we have $\mu(\gamma(\R))=0$ where $\gamma=\pi\circ\tilde\gamma$. 
In the Riemannian case, we have, equivalently, $\tilde\mu(\tilde\gamma(\R))=0$.
\end{enumerate}
\end{theorem} 

In the Riemannian case, Theorem \ref{thm2} implies that, under \ref{MM}, any microlocal QL is invariant under both horizontal and vertical geodesic flows (in addition to being invariant under the usual geodesic flow, as it is well known), and moreover does not charge any geodesic. 
In particular, any countable convex combination of Dirac measures of periodic geodesics on $S^*M$ (resp., of projections of geodesics on $M$) cannot be a microlocal QL of $S^*M$ (resp., a local QL of $M$).
This is in strong contrast with the case where the geodesic flow is completely integrable (see \cite{CdV_Duke1979,Toth}) or with the case of a revolution surface (with boundary) of constant negative curvature (see \cite{CdVParisse}) where the Dirac of the equatorial unstable periodic geodesic is a microlocal QL. 

\begin{remark}\label{rem_boundary}
Item \ref{itemiv} of Theorem \ref{thm2} is adapted as follows in the Riemannian case, when $M_1$ and/or $M_2$ are Riemannian manifolds having a smooth boundary: assuming that $\triangle_\ell$ is the Laplace-Beltrami operator on $M_\ell$ with Dirichlet or Neumann boundary conditions if $\partial M_\ell\neq\emptyset$ (for instance, one can take on $M_1$ the Dirichlet Laplacian and on $M_2$ the Neumann Laplacian), Item \ref{itemiv} remains true for any \emph{generalized} geodesic. 

We recall that, unformally, a generalized geodesic is a broken geodesic propagating in $S^*M$, reflecting at the boundary according to the laws of classical optics, and being a usual geodesic in the interior. Rigorously, a generalized geodesic is a projection of a generalized bicharacteristic, that is, a solution of the Melrose-Sj\"ostrand compressed bicharacteristic flow (see \cite[Chapter 24.3]{Hormander3}, \cite{MelroseSjostrand_CPAM1978} or \cite[Section 1C1]{LLTP}).
%
\end{remark}


Theorems \ref{thm1} is proved in Section \ref{sec_proof1}. Theorem \ref{thm2} and Remark \ref{rem_boundary} are proved in Section \ref{sec_proof2}.
It is interesting to note that, in Theorem \ref{thm2}, Item \ref{itemiv} does not follow from Item \ref{itemiii}. More precisely, it easily follows from Item \ref{itemiii} (by contradiction) that a microlocal QL cannot charge any periodic Hamiltonian curve that is neither horizontal nor vertical. But Item \ref{itemiii} gives no contradiction whenever $\tilde\mu=\tilde\rho_1\otimes\delta_{\tilde\gamma_2}$ (and thus $\mu=\rho_1\otimes\delta_{\gamma_2}$) where $\tilde\rho_2$ is the normalized Liouville measure on $S^*M_2$, $\tilde\delta_{\gamma_2}$ is the Dirac along a periodic Hamiltonian curve $\tilde\gamma_2$ on $S^*M_2$, and $\gamma_2=\pi_2\circ\tilde\gamma_2$. Indeed, such a measure $\tilde\mu$ is invariant under the three Hamiltonian flows, and $\mu$ obviously charges $\gamma$.
To discard this case and thus to prove that $\mu$ cannot be a local QL, we develop another proof strategy (see Section \ref{sec_thm2_2}).

Note that, in Section \ref{sec_additional}, we establish some additional properties that are not reported in Theorem \ref{thm2}.

The fact, stated in Item \ref{itemiv}, that $\mu(\gamma(\R))=0$ if and only if $\tilde\mu(\tilde\gamma(\R))=0$ in the Riemannian case, follows from Lemma \ref{lemraynul} in Appendix \ref{app:measures}. We provide in this appendix several facts for QLs in the Riemannian case.

\subsection{Comments on the assumptions}

Assumption \ref{UC} is satisfied if, for $\ell=1,2$, $M_\ell$ is a smooth connected manifold and $\triangle_\ell$ is an elliptic second-order operator whose coefficients are uniformly bounded and with coefficients of the principal part being of class $C^2$ and with Lipschitzian second derivatives (see \cite{Aronszajn}). 
It is also satisfied if $M_\ell$ is an analytic manifold and $\triangle_\ell$ is a sum of squares of analytic vector fields satisfying the H\"ormander assumption (subelliptic H\"ormander operator, see \cite{Bony}).

It follows from Sobolev embedding theorems that Assumption \ref{B} is satisfied if, for $\ell=1,2$, $M_\ell$ is a smooth manifold and $\triangle_\ell$ is an elliptic second-order operator with smooth coefficients, or a sum of squares of smooth vector fields satisfying the H\"ormander assumption (see \cite{RothschildStein}).

Assumption \ref{MM} is equivalent to:
\begin{equation} \label{simplicity2}
\forall (i,j,i',j')\in(\N^*)^4\qquad
\lambda_{i,j}=\lambda_{i',j'} \ \Rightarrow\  (i=i' \; \hbox{ and } j=j') 
\end{equation}
i.e., the family $(\lambda_{i,j})_{(i,j)\in(\N^*)^2}$ consists of the \emph{distinct} eigenvalues of $\triangle$.
Assumption \ref{MM} is generic in some sense, as established in the following result. 

\begin{proposition} \label{dilatation}
Let $\alpha\in\R\setminus\{0\}$ and let $(\triangle_{2,s})_{s>0}$ be a family of self-adjoint operators on $M_2$ such that $\triangle_{2,1}=\triangle_2$ and the $j^{\textrm{th}}$ (distinct) eigenvalue of $\triangle_{2,s}$ is $s^\alpha \lambda^2_j$. Then, the set of $s>0$ such that the spectrum of $\triangle_1 + \triangle_{2,s}$ does not satisfy \ref{MM} is at most countable. 
\end{proposition}

Proposition \ref{dilatation} is proved in Appendix~\ref{proof:propdilat}. It applies for example to the case where $\triangle_{2,s}$ is the Laplace-Beltrami operator for the metric $s g_2$ (with $\alpha=-1$).

Assumption \ref{MM} is not satisfied in the Dirichlet square: take $M_\ell=[0,\pi]$ endowed with the Lebesgue measure and with the standard Dirichlet Laplacian, so that $\lambda^\ell_i=i^2$ for every $i\in\N^*$, for $\ell=1,2$. Then, $\triangle$ is the standard Dirichlet Laplacian in the square $M=[0,\pi]^2$, of eigenvalues $\lambda_{i,j}=i^2+j^2$, for $i,j\in\N^*$. The minimal multiplicity property \ref{MM} is not satisfied because of the existence of positive integers that can be written in more than one way as the sum of two squares of positive integers. An example is $50=1^2+7^2=5^2+5^2$. Actually, given any $p\in\N^*$, there exist positive integers that can be written in $p$ different ways as the sum of two squares of positive integers. Hence, although the 1D Dirichlet Laplacian satisfies $\displaystyle \inf_{\mu\in\QL([0,\pi])}\mu(\omega)>0$ for any measurable subset $\omega$ of $[0,\pi]$ of positive Lebesgue measure\footnote{This follows from the inequality $$\inf_{j\in\N^*} \int_\omega \sin^2(jx)\, dx \geq \frac{1}{2}(\pi\vert\omega\vert-\sin(\pi\vert\omega\vert))$$ for any $\omega\subset[0,\pi]$ measurable, see \cite[Lemma 6]{PTZ_HUM}.},
Theorem \ref{thm1} does not apply to that case. Yet, we have the following result.

\begin{lemma}\label{lem_square}
For every open subset $\omega$ of $(0,\pi)^2$, there exists $C_\omega>0$ such that
\begin{equation}\label{conj_grebenkov}
\int_\omega \vert\phi\vert^2\, dx   \geq   C_\omega  \int_{[0,\pi]^2} \vert\phi\vert^2\, dx
\end{equation}
for any eigenfunction of the standard Dirichlet Laplacian on $[0,\pi]^2$.
As a consequence, the conclusion of Item \ref{itemi} of Theorem \ref{thm1} is satisfied for the Dirichlet Laplacian in the square: $\displaystyle\inf_{\mu \in \QL(M) } \mu(\omega)>0$ for any open subset $\omega$ of $M=[0,\pi]^2$.
\end{lemma}

Lemma \ref{lem_square} is proved in Appendix~\ref{proof:lem_square}.
Note that \eqref{conj_grebenkov} gives a positive answer to the open problem raised in \cite[Section 7.7.3]{GrebenkovNguyen} and \cite[Section 4.1]{NguyenGrebenkov}.

Note that, thanks to Proposition \ref{dilatation}, if instead of $M_2=[0,\pi]$ we take $M_2=[0,\pi/s]$ with $s>0$ (rectangle domain), then \ref{MM} is satisfied for every $s>0$ except for a countable number of exceptional values (including $s=1$).

\subsection{Comments on QLs}
There are quite few existing results on regularity properties of QLs. The properties studied in Items \ref{itemi} (local QLs being of full support) and \ref{itemii} (local QLs being absolutely continuous) of Theorem \ref{thm1} are rather exceptional.

In 1D, obviously, the Dirichlet Laplacian has a unique microlocal QL that is the Liouville measure: we have quantum unique ergodicity (QUE). The QUE property is also valid on any compact arithmetic surface (see \cite{Lindenstrauss}). To date, these are the only cases where QUE is known.

In \cite{Jakobson}, the author shows (among other results) that, on the flat torus of any dimension, all local QLs of the Laplace-Beltrami operator are absolutely continuous. 
In \cite{PTZ_JEMS}, all local QLs of the Dirichlet Laplacian in the disk are computed, and they are all absolutely continuous except one which is the Dirac along the boundary.

On the Euclidean sphere (see \cite{JakobsonZelditch}) or more generally on Zoll manifolds with maximally degenerate Laplacian (see \cite{Macia,Zelditch1996}), any Radon probability measure on the co-sphere bundle that is invariant under the geodesic flow is a microlocal QL of the Laplace-Beltrami operator; in particular, the Dirac along any (periodic) geodesic is a microlocal QL. 

In \cite{Anantharaman2008}, it is proved that on a negatively curved compact Riemannian manifold, every microlocal QL has positive metric entropy and thus cannot be the Dirac along a periodic geodesic (see also \cite{AnantharamanNonnenmacher,Nonnenmacher_AIF2017}).
It is proved in \cite{DyatlovLin} that, on any compact connected Riemannian surface of constant negative curvature, any microlocal QL has full support, i.e., charges any nonempty open subset of the cosphere bundle.

While all above results are for Riemannian (elliptic) Laplacians, decomposition results for microlocal QLs are given in \cite[Theorem B]{CHT_Duke} for sub-Riemannian (subelliptic) Laplacians and in \cite{Letrouit} more particularly for flat Heisenberg cases.

\section{Preliminary results on some spectral functionals}
\subsection{Definition and first properties}
As at the beginning of Section \ref{sec_setting}, we consider a smooth compact manifold $M$ endowed with a probability measure $\rho$ and a self-adjoint operator $\triangle$ on $L^2(M,\rho)$. We denote by $\mathcal{E}$ the set of all normalized eigenfunctions $\phi_\lambda$ of $\triangle$ in $L^2(M,\rho)$, i.e., $\triangle\phi_\lambda=\lambda\phi_\lambda$ and $\Vert\phi_\lambda\Vert_{L^2(M,\rho)}=1$ for every $\lambda\in\mathrm{Spec}(\triangle)$, and by $\QL(M)$ the set of local QLs on $M$.

\begin{definition}
Given any nonnegative bounded measurable function $a$ on $(M,\rho)$, we define
\begin{equation}\label{defg}
\g_M(a) =\inf_{\phi\in\mathcal{M}} \int_M a \vert\phi\vert^2 \, d\rho
\end{equation}
and
\begin{equation}\label{defgprime}
\g'_M(a) =\inf_{\mu\in\QL(M)} \mu(a) .
\end{equation}
\end{definition}

By the definition \eqref{defg} of $\g_M$, for every eigenfunction $\phi$ of $\triangle$, we have
\begin{equation}\label{inegg}
\int_M a \vert\phi\vert^2 \, d\rho \geq \g_M(a) \Vert\phi\Vert_{L^2(M,\rho)}^2 .
\end{equation}
When $a=\mathds{1}_\omega$ is the characteristic function\footnote{By definition, $\mathds{1}_\omega(x)$=1 if $x\in\omega$, and $0$ otherwise.} of some measurable subset $\omega\subset M$, we write $\g_M(\omega)$ instead of $\g_M(\mathds{1}_\omega)$ and we note that $\g_M(\omega)\in[0,1]$ and that $\g_M(M)=1$ (and similarly for $\g'_M$).

The spectral functional $\g_M$ has been introduced in \cite{hpt} (with the notation $g_1$) in order to investigate observability properties of wave equations.

\begin{lemma}\label{lemggprime}
We have $\g_M(a)\leq\g'_M(a)$ for every nonnegative bounded Borel function $a$ on $(M,\rho)$ for which the $\mu$-measure of the set of discontinuities of $a$ is zero for every $\mu\in\QL(M)$.
We have $\g_M(\omega)\leq \g'_M(\omega)$ for every closed subset $\omega$ of $M$.

Given any open subset $\omega$ of $M$, under \ref{UC}, if $\g_M(\omega)=0$ then $\g'_M(\omega)=0$.
\end{lemma}

\begin{proof}
The two first claims follow from the Portmanteau theorem (see Appendix \ref{cintre}) and by definition of local QLs. Let us prove the third claim. If $\g_M(\omega)=0$ for some $\omega\subset M$ open, then either $\int_\omega\vert\phi\vert^2\,d\rho=0$ for some $\phi\in\mathcal{E}$ or $\liminf_{\lambda\rightarrow+\infty}\int_\omega\vert\phi_\lambda\vert^2\,d\rho=0$. The first case cannot happen because \ref{UC} would imply that $\phi=0$. Therefore $\liminf_{\lambda\rightarrow+\infty}\int_\omega\vert\phi_\lambda\vert^2\,d\rho=0$. Taking a subsequence if necessary, there exists $\mu\in\QL(M)$ that is a weak limit of $\vert\phi_\lambda\vert^2\,\rho$. Since $\omega$ is open, it follows from the Portmanteau theorem (see Appendix \ref{cintre}; note that $M$ is metrisable) that $\mu(\omega)=0$. The lemma is proved.
\end{proof}

\subsection{On a product manifold}
Still with the notations introduced in Section \ref{sec_setting}, we now consider the product manifold $M=M_1 \times M_2$, and we consider the spectral functionals $\g_M$, $g_{M_1}$ and $g_{M_2}$.
Throughout the article, we use the following notation: given any measurable subset $\omega$ of $M=M_1 \times M_2$, given any $x_1\in M_1$, we define
$$
\omega_{x_1}=\{ x_2 \in M_2 \ \mid\ (x_1,x_2) \in \omega\} = \omega \cap ( \{x_1\}\times M_2 )
$$
(vertical trace of $\omega$ above $x_1$).


\begin{proposition} \label{main2} 
Assume that the spectrum of $\triangle = \triangle_1\otimes\mathrm{id}_{M_2} + \mathrm{id}_{M_1}\otimes\triangle_2$ satisfies \ref{MM}. 
Then, for every measurable subset $\omega$ of $M=M_1 \times M_2$, we have 
$$
\g_M(\omega) \geq \g_{M_1} ( x_1\mapsto\g_{M_2}(\omega_{x_1}) ) .
$$
\end{proposition} 

\begin{proof}
For $\ell=1,2$, we fix a Hilbert basis $(\phi^\ell_{i,k})_{i\in\N,\, k\in\{1,\ldots,m^\ell_i\}}$ of $L^2(M_\ell,\rho_\ell)$ consisting of eigenfunctions of $\triangle_\ell$, as in \eqref{Hilbertbasis}.

It follows from Assumption \ref{MM} that, for any pair $(i,j)\in(\N^*)^2$, the functions $\phi^1_{i,k} \otimes \phi^2_{j,l}$, for $k\in\{1,\ldots,m^1_i\}$ and $l\in\{1,\ldots,m^2_j\}$, form an orthonormal basis of the eigenspace associated to the eigenvalue $\lambda_{i,j}=\lambda^1_i +\lambda^2_j$ of $\triangle$ (this is because, using \eqref{simplicity2}, there is no other pair $(i',j')\in(\N^*)^2$ such that $\lambda_{i,j}=\lambda^1_{i'} +\lambda^2_{j'}$). This remark is crucial for the arguments hereafter.

In what follows, we identify the function $x_1\mapsto\phi_i^1(x_1)$ (resp., $x_2\mapsto\phi_l^2(x_2)$) defined on $M_1$ (resp., on $M_2$) to its natural extension to $M$, namely the function $(x_1,x_2)\mapsto\phi_i^1(x_1)$ (resp., $(x_1,x_2)\mapsto\phi_l^2(x_2)$), so that $\phi_i^1 \otimes \phi_l^2$ is identified to the product $\phi_i^1  \phi_l^2$.

Let $\phi$ be an arbitrary eigenfunction of $\triangle$, of norm $\Vert\phi\Vert_{L^2(M,\rho)}=1$. There exists a (unique) pair $(i,j)\in(\N^*)^2$ such that $\phi$ is associated with the eigenvalue $\lambda_{i,j}$. By the above remark, $\phi$ can be expanded as
\begin{equation}\label{phiexpanded}
\phi= \sum_{k=1}^{m^1_i} \sum_{l=1}^{m^2_j} a_{k,l}\, \phi^1_{i,k} \phi^2_{j,l} = \sum_{l=1}^{m^2_j} b_{i,l} \phi^2_{j,l}  \qquad  \textrm{with}\qquad \sum_{k=1}^{m^1_i} \sum_{l=1}^{m^2_j} \vert a_{k,l} \vert^2 = 1 
\end{equation}
where we have set $b_{i,l} = \sum_{k=1}^{m^1_i} a_{k,l}\, \phi^1_{i,k}$.
Note that, for every $l\in\{1,\ldots,m^2_j\}$, the function $b_{i,l}$ does not depend on the variable $x_2\in M_2$ and is an eigenfunction of $\triangle_1$ associated to the eigenvalue $\lambda_i^1$, of norm given (using that the family $(\phi^1_{i,k})_{k\in\{1,\ldots,m^1_i\}}$ is orthonormal in $L^2(M_1,\rho_1)$) by
\begin{equation}\label{normbjl}
\Vert b_{i,l}\Vert_{L^2(M_1,\rho_1)}^2 = \sum_{k=1}^{m^1_i} \vert a_{k,l}\vert^2 . 
\end{equation}
Let $\omega \subset M$ be a measurable subset. By the Fubini theorem, we have
\begin{equation}\label{eq_fubini}
\int_\omega \vert \phi\vert^2 \, d\rho = \int_{M_1} \int_{\omega_{x_1}} \left\vert \sum_{l=1}^{m^2_j} b_{i,l}(x_1) \phi^2_{j,l}(x_2) \right\vert^2  d\rho_2(x_2) \, d\rho_1(x_1) .
\end{equation}
Noting that, for every fixed $x_1\in M_1$, the function $\sum_{l=1}^{m^2_j} b_{i,l}(x_1) \phi^2_{j,l}$ is an eigenfunction of $\triangle_2$ associated to the eigenvalue $\lambda^2_j$, of $L^2$ norm given (using that the family $(\phi^2_{j,l})_{l\in\{1,\ldots,m^2_j\}}$ is orthonormal in $L^2(M_2,\rho_2)$) by 
$$
\left\Vert \sum_{l=1}^{m^2_j} b_{i,l}(x_1) \phi^2_{j,l} \right\Vert_{L^2(M_2,\rho_2)}^2 = \sum_{l=1}^{m^2_j} \vert b_{i,l}(x_1)\vert^2   ,
$$ 
we infer from \eqref{inegg} (with $a=(x_2\mapsto\mathds{1}_{\omega_{x_1}}(x_2))$ on $M_2$) that
\begin{equation}\label{minor2}
\int_{\omega_{x_1}} \left\vert \sum_{l=1}^{m^2_j} b_{i,l}(x_1) \phi^2_{j,l}(x_2) \right\vert^2  d\rho_2(x_2) \geq \g_{M_2}(\omega_{x_1}) \sum_{l=1}^{m^2_j} \vert b_{i,l}(x_1)\vert^2   \qquad\forall x_1\in M_1.  
\end{equation}
It follows from \eqref{eq_fubini} and \eqref{minor2} that
\begin{equation}\label{18:00}
\int_\omega \vert \phi\vert^2 dm  \geq \sum_{l=1}^{m^2_j} \int_{M_1} \g_{M_2}(\omega_{x_1}) \vert b_{i,l}(x_1)\vert^2\, d\rho_1(x_1) .
\end{equation}
For every $l\in\{1,\ldots,m^2_j\}$, since $b_{i,l}$ is an eigenfunction of $\triangle_1$ associated to the eigenvalue $\lambda^1_i$, of $L^2$ norm given by \eqref{normbjl}, we infer from \eqref{inegg} (with $a=(x_1\mapsto\g_{M_2}(\omega_{x_1}))$ on $M_1$) that
\begin{equation}\label{18:01}
\int_{M_1} \g_{M_2}(\omega_{x_1}) \vert b_{i,l}(x_1)\vert^2\, d\rho_1(x_1) \geq \g_{M_1} ( x_1\mapsto\g_{M_2}(\omega_{x_1}) ) \sum_{k=1}^{m^1_i} \vert a_{k,l}\vert^2  . 
\end{equation}
Summing \eqref{18:01} over $l\in\{1,\ldots,m^2_j\}$ and using the right-hand equality of \eqref{phiexpanded}, we infer from \eqref{18:00} that 
 $$
 \int_\omega \vert \phi\vert^2 dm  \geq \g_{M_1} ( x_1\mapsto\g_{M_2}(\omega_{x_1}) ) .
 $$
Since $\phi$ is an arbitrary normalized eigenfunction of $\triangle$, using \eqref{inegg}, the conclusion follows. 
\end{proof}


\begin{remark}
In some sense, the argument of the above proof generalizes the Fubini argument of \cite[Proof of Theorem 4.5]{Laurent_MCRF2014}, noticed in \cite{BurqZworski} to prove that, on a product of two Riemannian manifolds, the Schr\"odinger equation is observable on $\omega_1\times M_2$ in time $T$ for any open subset $\omega_1$ of $M_1$ satisfying the geometric control condition on $M_1$ in time $T$. 
However, it does not seem that Proposition \ref{main2} (nor Theorem \ref{thm1}) can yield significant new results as concerns observability issues.
\end{remark}

\section{Proof of Theorem \ref{thm1}}\label{sec_proof1}
\subsection{Proof of \ref{itemi}}
We have to prove that, if $\displaystyle\inf_{\mu \in \QL(M_\ell) } \mu(\omega_\ell)>0$ for any open subset $\omega_\ell \subset M_\ell$, for $\ell=1,2$, then $\displaystyle\inf_{\mu \in \QL(M) } \mu(\omega)>0$ for any open subset $\omega \subset M$.

Let $\omega$ be an open subset of $M$. 
There exists two open balls $B_1 \subset M_1$ and $B_2 \subset M_2$ such that $B_1 \times B_2\subset\omega$.  By assumption, we have
\begin{equation} \label{assump}
\g'_{M_1}(B_1) = \inf_{\mu \in \QL(M_1) } \mu(B_1)>0\qquad \textrm{and}\qquad \g'_{M_2}(B_2) = \inf_{\mu \in \QL(M_2) } \mu(B_2)>0.
\end{equation}
Since $B_2$ is open, using \ref{UC} it follows from the last item of Lemma \ref{lemggprime} that $\g_{M_2}(B_2) >0$. 


Since $B_2 \subset \omega_{x_1}$ for every $x_1 \in B_1$, the function $x_1 \mapsto \g_{M_2}(\omega_{x_1})$ is bounded below, on $M_1$, by the function $\g_{M_2}(B_2)\mathds{1}_{B_1}$. Using \eqref{assump}, we infer that 
$$
\g_{M_1}(x_1\mapsto\g_{M_2}(\omega_{x_1})) \geq \g_{M_1}(B_1) \g_{M_2}(B_2) >0 .
$$
Using \ref{MM} and applying Proposition~\ref{main2}, we obtain $\g_M(\omega) \geq \g_{M_1}(B_1) \g_{M_2}(B_2)$. The conclusion follows.

\subsection{Proof of \ref{itemii}}\label{sec_proof_ii}

We assume that every $\mu_\ell \in \QL(M_\ell)$ is absolutely continuous with respect to $\rho_\ell$.

\begin{lemma} \label{step1} 
Let $\ell\in\{1,2\}$. Let $(A_\varepsilon)_{\varepsilon>0}$ be a nonincreasing family of nested closed sets in $M_\ell$, i.e., such that $A_\varepsilon \subset A_{\varepsilon'}$ if $0< \varepsilon < \varepsilon'$ and  $\rho_\ell(A_\varepsilon)\rightarrow 0$ as $\varepsilon\rightarrow 0$. Then
$$
\lim_{\varepsilon \rightarrow 0} \g_{M_\ell}(M \setminus A_\varepsilon)= 1.
$$
\end{lemma}

\begin{proof} 
Without loss of generality, we assume that $A_\varepsilon\neq\emptyset$ for every $\varepsilon>0$.
By contradiction, assume that there exists $\delta>0$ such that, for every $\varepsilon> 0$, there exists an eigenfunction $\phi_\varepsilon$ of $\triangle_\ell$, of norm $1$ in $L^2(M_\ell,\rho_\ell)$, associated with the eigenvalue $\lambda^\ell_{j(\varepsilon)}$ (for some $j(\varepsilon)\in\N^*$), such that $\int_{A_\varepsilon} \vert \phi_\varepsilon\vert^2 d\rho_\ell \geq \delta$. 

We claim that $\lambda^\ell_{j(\varepsilon)}\rightarrow+\infty$ as $\varepsilon\rightarrow 0$. Indeed, otherwise, there exists a sequence $\varepsilon_k\rightarrow 0$ such that $(\lambda^\ell_{j(\varepsilon_k)})_{k\in\N^*}$ is bounded, and it follows from \ref{B} that there exists $C>0$ such that $\Vert\phi_{\varepsilon_k}\Vert_{L^\infty(M_\ell,\rho_\ell)}\leq C$ for every $k\in\N^*$. 
We infer that 
$$
C^2\rho_\ell(A_{\varepsilon_k})\geq \rho_\ell(A_{\varepsilon_k}) \Vert\phi_{\varepsilon_k}\Vert_{L^\infty(M_\ell,\rho_\ell)}^2 \geq \int_{A_{\varepsilon_k}} \vert \phi_{\varepsilon_k}\vert^2 d\rho_\ell \geq \delta
$$ 
which contradicts the fact that $\rho_\ell(A_{\varepsilon_k})\rightarrow 0$ as $k\rightarrow +\infty$, and proves the claim.

Since $\lambda^\ell_{j(\varepsilon)}\rightarrow+\infty$, by definition of quantum limits, there exist $\mu \in \QL(M_\ell)$ and a sequence $\varepsilon_k\rightarrow 0$ such that the sequence of probability measures $\vert\phi_{\varepsilon_k}\vert^2 \, \rho_\ell$ converges (weakly) to $\mu$. 
Given any $k\in\N^*$ and $\varepsilon>0$ such that $\varepsilon> \varepsilon_k>0$, since $A_{\varepsilon_k} \subset A_\varepsilon$, we have $\int_{A_\varepsilon} \vert \phi_{\varepsilon_k}\vert^2 d\rho_\ell \geq \delta$.
Since $A_\varepsilon$ is closed, passing to the limit $k\rightarrow+\infty$ and using the Portmanteau  theorem (see Appendix~\ref{cintre}), we infer that $\mu(A_\varepsilon ) \geq \delta$ for every $\varepsilon>0$. Defining $x\in M_\ell$ such that $\{x\} = \cap_{\varepsilon>0} A_\varepsilon$, we thus obtain $\mu(\{x\})\geq \delta$, while $\rho_\ell(\{x\})=0$ because $\rho_\ell(A_\varepsilon)\rightarrow 0$ as $\varepsilon\rightarrow 0$. This contradicts the fact that $\mu$ must be absolutely continuous with respect to $\rho_\ell$ by assumption. The lemma is proved.
\end{proof}

Let $\mu\in\QL(M)$ and let $A$ be a $\rho$-measurable subset $A$ of $M$ such that $\rho(A)=0$. We want to prove that $\mu(A)=0$. Since $\mu$ is regular (as a Borel measure on a compact metric space), we have $\mu ( A ) = \sup \{ \mu (F) \ \mid\ F \subset A , \ F\text{  closed}\}$ and thus, without loss of generality, we assume that $A$ is closed and thus compact. For every $\varepsilon >0$, we define the open subset $\omega_\varepsilon$ and the closed subset $F_\varepsilon$ of $M$ by
$$
\omega_\varepsilon = \{x \in M \ \mid\ d_M(x,A) < \varepsilon\} , \qquad F_\varepsilon = M\setminus \omega_\varepsilon ,
$$
where $d_M$ is an arbitrary distance on $M$ inducing its topology.
Note that $A=\cap_{\varepsilon>0}\omega_\varepsilon=\cap_{\varepsilon>0}\overline{\omega}_\varepsilon$.

We claim that $\rho(\omega_\varepsilon)\rightarrow 0$ as $\varepsilon\rightarrow 0$.
Indeed, since $\rho$ is regular (as a Borel measure on a compact metric space), we have $\rho( A ) = \inf \{ \rho (O) \ \mid\ A \subset O , \ O\text{  open}\}$. Hence, for every $\eta>0$, there exists an open subset $A_\eta$ of $M$ such that $A \subset A_\eta$ and $\rho(A_\eta) \leq \eta$. Since $M \setminus A_\eta$ and $A$ are compact, we have $d_M(A,M \setminus A_\eta)>0$. Taking $0<\varepsilon <d_M(A,M \setminus A_\eta)$, we have $\omega_\varepsilon \subset A_\eta$ and therefore $\rho(\omega_\varepsilon) \leq \rho(A_\eta) \leq \eta$. The claim follows.

Given any $\delta>0$ and $\varepsilon>0$, we define 
$$
\Omega_\varepsilon^\delta = \left\{ x_1 \in M_1 \ \mid\  \g_{M_2}((F_\varepsilon)_{x_1}) \leq 1-\delta \right\} .
$$

\begin{lemma} \label{step2}
For every $\delta>0$, we have $\rho_1( \overline{\Omega}_\varepsilon^\delta ) \rightarrow 0$ as $\varepsilon\rightarrow 0$. 
\end{lemma}

\begin{proof}
Since the family $(F_\varepsilon)_{\varepsilon>0}$ is nonincreasing for the inclusion with respect to $\varepsilon$, the family $(\Omega_\varepsilon^\delta)_{\varepsilon>0}$, and thus as well $(\overline\Omega_\varepsilon^\delta)_{\varepsilon>0}$, is nondecreasing for the inclusion, i.e., $\Omega_{\varepsilon'}^\delta \subset \Omega_\varepsilon^\delta$ if $0<\varepsilon'<\varepsilon$. 

We claim that $\cap_{\varepsilon >0} \overline{\Omega}_\varepsilon^\delta = \cap_{\varepsilon >0} \Omega_\varepsilon^\delta$.
Indeed, let $x_1 \in \cap_{\varepsilon > 0} \overline{\Omega}_\varepsilon^\delta$. 
Given any $0<\varepsilon' <\varepsilon$, the point $x_1 \in \overline{\Omega}_{\varepsilon'}$ is the limit in $M_1$ of a sequence of points $z_k \in \Omega_{\varepsilon'}^\delta$, which satisfy, by definition, $\g_{M_2}((F_{\varepsilon'})_{z_k}) \leq 1-\delta$.
We have $(F_\varepsilon)_{x_1} \subset (F_{\varepsilon'})_{z_k}$ whenever $k$ is large enough (indeed, if $(x_1,x_2)\in F_\varepsilon$ then $d_M((x_1,x_2),A)\geq\varepsilon>\varepsilon'$ and thus $d_M((z_k,x_2),A)\geq\varepsilon'$ for $k$ large enough, i.e., $(z_k,x_2)\in F_{\varepsilon'}$), hence $\g_{M_2}((F_{\varepsilon})_{x_1})  \leq \g_{M_2}((F_{\varepsilon'})_{z_k}) \leq 1 - \delta$
and thus $x_1 \in \Omega_\varepsilon^\delta$. The claim is proved.

Then, to prove the lemma, equivalently, let us prove that $\rho_1( \cap_{\varepsilon >0} \Omega_\varepsilon^\delta ) = 0$.

Let $x_1 \in \cap_{\varepsilon >0} \Omega_\varepsilon^\delta$ be arbitrary. For every $\varepsilon>0$, by definition, we have $\g_{M_2}((F_\varepsilon)_{x_1} ) \leq 1-\delta$. 
Since $(F_\varepsilon)_{x_1}=M_2 \setminus (\omega_\varepsilon)_{x_1}$ and $(\mathring{F}_\varepsilon)_{x_1}=M_2 \setminus (\overline{\omega}_\varepsilon)_{x_1}$, 
the closed subset $A_\varepsilon = (\overline{\omega}_{\varepsilon})_{x_1}$ of $M_2$ is such that $M_2 \setminus A_\varepsilon = (\mathring{F}_\varepsilon)_{x_1} \subset (F_\varepsilon)_{x_1}$, hence $\g_{M_2} (M_2 \setminus A_\varepsilon) \leq  \g_{M_2} ((F_\varepsilon)_{x_1}) \leq 1- \delta$.
Since $(A_\varepsilon)_{\varepsilon>0}$ is a nonincreasing family of nested closed sets in $M_2$, it follows from Lemma~\ref{step1}, by contraposition, that $\lim_{\varepsilon \rightarrow 0} \rho_2(A_\varepsilon) = \rho_2(\cap_{\varepsilon >0} A_\varepsilon) >0$.
%
It is easy to see that
$$
\bigcap_{\varepsilon>0} A_\varepsilon = \bigcap_{\varepsilon >0} (\overline{\omega}_\varepsilon)_{x_1} = \left( \bigcap_{\varepsilon >0} \overline{\omega}_\varepsilon \right)_{x_1} 
= A_{x_1} = \{x_2 \in M_2 \ \mid\  (x_1,x_2) \in A \}. 
$$
We have therefore obtained that $\rho_2(A_{x_1}) >0$ for every $x_1 \in \cap_{\varepsilon >0} \Omega_\varepsilon^\delta$. 
Since
$$
0 = \rho(A) = \int_{M_1} \rho_2(A_{x_1})\, d\rho_1(x_1) ,
$$
we must have $\rho_2(A_{x_1})=0$ for $\rho_1$-almost every $x_1\in M_1$. Therefore $\rho_1( \cap_{\varepsilon >0} \Omega_\varepsilon^\delta ) = 0$. The lemma is proved.
\end{proof}

\begin{lemma} \label{step3} 
We have $\g_M(F_\varepsilon) \rightarrow 1$ as $\varepsilon \rightarrow 0$.
\end{lemma}

\begin{proof}
By Assumption \ref{MM} and Proposition~\ref{main2}, it suffices to prove that 
$$
\lim_{\varepsilon \rightarrow 0} \g_{M_1}(x_1\mapsto\g_{M_2}((F_\varepsilon)_{x_1})) =1 .
$$
Noting that, by definition of $\Omega_\varepsilon^\delta$, $ \g_{M_2}((F_\varepsilon)_{x_1}) \geq 1 -\delta$ for every $x_1\in M_1 \setminus \overline{\Omega}_\varepsilon^\delta$, we have
\begin{equation*}
\begin{split}
\g_{M_1}(x_1\mapsto \g_{M_2}((F_\varepsilon)_{x_1})) 
&= \inf_{\phi\in\mathcal{E}_1} \int_{M_1} \vert \phi(x_1)\vert^2\, \g_{M_2}((F_\varepsilon)_{x_1}) \, d\rho_1(x_1) \\
&\geq \inf_{\phi\in\mathcal{E}_1} \int_{M_1\setminus \overline{\Omega}_\varepsilon^\delta} \vert \phi(x_1)\vert^2\, \g_{M_2}((F_\varepsilon)_{x_1}) \, d\rho_1(x_1) \\
& \geq (1-\delta) \inf_{\phi\in\mathcal{E}_1} \int_{M_1\setminus \overline{\Omega}_\varepsilon^\delta} \vert \phi(x_1)\vert^2 d\rho_1(x_1) 
= (1- \delta) \g_{M_1}(M_1 \setminus \overline{\Omega}_\varepsilon^\delta) .
\end{split}
\end{equation*}
By Lemmas~\ref{step1} and \ref{step2}, we have $\g_{M_1}(M_1 \setminus \overline{\Omega}_\varepsilon^\delta)\rightarrow 1$ as $\varepsilon \rightarrow 0$, and thus 
$$
\liminf_{\varepsilon \rightarrow 0} \g_{M_1}(x_1\mapsto \g_{M_2}((F_\varepsilon)_{x_1}))  \geq 1-\delta .
$$
Since $\delta>0$ is arbitrary, the lemma follows. 
\end{proof}

Let us finally prove that $\mu(A)=0$.
Since $F_\varepsilon$ is closed, it follows from Lemma \ref{lemggprime} that $\g_M(F_\varepsilon)\leq \g'_M(F_\varepsilon)$. Besides, by the definition \eqref{defgprime} of $\g'_M$, we have $\g'_M(F_\varepsilon) \leq \mu(F_\varepsilon)$. Using Lemma~\ref{step3} we infer that $\mu(F_\varepsilon)\rightarrow 1$ as $\varepsilon \rightarrow 0$ and thus, since $\omega_\varepsilon = M \setminus F_\varepsilon$, $\mu(\omega_\varepsilon)\rightarrow 0$ as $\varepsilon \rightarrow 0$. Since $A=\cap_{\varepsilon>0}\omega_\varepsilon$, we conclude that $\mu(A)=0$.  


\section{Proof of Theorem \ref{thm2}}\label{sec_proof2}
Let $\tilde\mu$ be a microlocal QL of $\triangle$ and let $\mu=\pi_*\tilde\mu\in\QL(M)$.

\subsection{Proof of \ref{itemiii}}
\paragraph{Preliminary remark.}
The operators $(\triangle_1)_{x_1} = \triangle_1\otimes\mathrm{id}$ and $(\triangle_2)_{x_2} = \mathrm{id}\otimes\triangle_2$ have a joint spectrum.
Moreover, thanks to \ref{MM} we have the following much stronger spectral property: any eigenfunction of $\triangle=(\triangle_1)_{x_1}+(\triangle_2)_{x_2}$ is actually an eigenfunction of both operators $(\triangle_1)_{x_1}$ and $(\triangle_2)_{x_2}$; more precisely, any eigenvalue $\lambda$ of $\triangle$ is written as $\lambda = \lambda^1 + \lambda^2$ where $\lambda^\ell$ is an eigenvalue of $\triangle_\ell$ on $M_\ell$, for $\ell=1,2$, this decomposition being unique by \ref{MM}, and
\begin{equation}\label{eigenfell}
(\triangle_\ell)_{x_\ell} \phi_\lambda = \lambda^\ell \phi_\lambda,\qquad \ell=1,2,
\end{equation}
for every eigenfunction $\phi_\lambda$ of $\triangle=(\triangle_1)_{x_1}+(\triangle_2)_{x_2}$ corresponding to the eigenvalue $\lambda = \lambda^1 + \lambda^2$. This can be seen by considering the tensorized Hilbert basis as at the beginning of Proposition \ref{main2}.

\paragraph{Invariance under the horizontal and vertical Hamiltonian flows.}
The above spectral property implies the invariance of any microlocal QL of $\triangle$ under both the horizontal and vertical Hamiltonian flows. The argument is quite classical (see \cite{CdV_Duke1979,Toth}) and goes as follows.

By definition of a QL, there exists a sequence of eigenvalues $\lambda$ of $\triangle$ (we do not write sequences to keep readable notations), with associated normalized eigenfunctions $\phi_\lambda$, such that $\tilde\mu_{\phi_\lambda}$ converges weakly to $\tilde\mu$ as $\lambda\rightarrow+\infty$, where $\tilde\mu_{\phi_\lambda}(a) = \langle\Op(a)\phi_\lambda,\phi_\lambda\rangle$ for every classical symbol $a$ of order $0$ and for an arbitrary quantization $\Op$.

It is already known, as a consequence of the (infinitesimal) Egorov theorem (see, e.g., \cite{Zworski}) that $\tilde\mu$ is invariant under the Hamiltonian flow of $S^*M$, i.e., $\tilde\mu( \{h,a\} ) = 0$ for every $a\in C^\infty(S^*M)$, 
where $h$ is the principal symbol (of order $d$) of $\triangle$. 
Let us anyway recall the argument. Let $A$ be a (classical) pseudodifferential operator of order $0$, with principal symbol $a$, microlocally supported away of the characteristic variety $\Sigma=\{h=0\}$ of $\triangle$. Since $\triangle\phi_\lambda = \lambda\phi_\lambda$, we have
\begin{equation}\label{1surd}
\big\langle [ \triangle^{1/d}, A ] \phi_\lambda, \phi_\lambda \big\rangle_{L^2(M,\rho)} = 0  .
\end{equation}
On the microlocal support of $A$ (thus, outside of $\Sigma$), $\triangle$ is elliptic (and nonnegative) and thus $\triangle^{1/d}$ is a pseudodifferential operator of order $1$, with principal symbol $h^{1/d}$ (see, e.g., \cite[Cor. 9]{HassellVasy}). Taking the limit $\lambda\rightarrow+\infty$ in \eqref{1surd}, we obtain $\tilde\mu( \{h,a\} ) = 0$, whence the invariance of $\tilde\mu$ under the Hamiltonian flow (note that the Hamiltonian flow is trivial in $\Sigma$).

Let us now prove that, thanks to \ref{MM}, $\tilde\mu$ is also invariant under the horizontal and the vertical Hamiltonian flows.
%
Recall that $h_\ell\geq 0$ is the principal symbol (of order $d$) of $\triangle_\ell$, for $\ell=1,2$, and that $h=h_1+h_2$.
Let $A$ be a pseudodifferential operator of order $0$, with principal symbol $a$, microlocally supported away of $\Sigma$.
As above, on the microlocal support of $A$,
$\triangle^{-(d-1)/d}$ is a pseudodifferential operator of order $1-d$, with principal operator $h^{-(d-1)/d}$.
Then, for $\ell=1,2$, the operator $\triangle^{-(d-1)/d} (\triangle_\ell)_{x_\ell}$ is a pseudodifferential operator of order $1$, with principal symbol $\frac{1}{i} \big\{\frac{h_\ell}{h^{(d-1)/d}},a\big\}$, and we have
$$
\big\langle [ \triangle^{-(d-1)/d} (\triangle_\ell)_{x_\ell}, A ] \phi_\lambda, \phi_\lambda \big\rangle_{L^2(M,\rho)} = 0  
, \qquad \ell=1,2.
$$
Now, taking the limit $\lambda\rightarrow+\infty$ yields $\tilde\mu\big( \big\{\frac{h_\ell}{h^{(d-1)/d}},a\big\} \big) = 0$, and since we already know that $\tilde\mu( \{h,a\} ) = 0$, we obtain that $\tilde\mu( \{h_\ell,a\} ) = 0$, for $\ell=1,2$, and thus $\tilde\mu$ is invariant under both the horizontal and vertical Hamiltonian flows.

\begin{remark}\label{rem_action}
Actually, the above argument shows that $\tilde\mu$ is invariant under the Hamiltonian flow of any convex combination of $h_1$ and $h_2$ (viewed as smooth functions on $T^*M$), thus, under an action of the quadrant $[0,+\infty)^2$, and horizontal and vertical Hamiltonian flows correspond to the boundary of this quadrant. The action is free inside the quadrant but not at the boundary.
\end{remark}

\begin{remark}[$\tilde\mu$ does not charge any non-horizontal and non-vertical Hamiltonian curve]\label{rem_nonhornonvert}
Given any Hamiltonian curve 
$\tilde\gamma(\cdot)=(\tilde\gamma_1(\cdot),\tilde\gamma_2(\cdot))$ 
that is neither horizontal nor vertical, 
we claim that $\tilde\mu(\tilde\gamma(\R))=0$; equivalently (by Lemma \ref{lemraynul} in Appendix \ref{app:measures}), $\mu(\gamma(\R))=0$, where $\gamma=\pi\circ\tilde\gamma$.

Indeed, by contradiction, if $\tilde\mu(\tilde\gamma((0,T)))>0$, since $\tilde\mu$ is invariant under both the horizontal and vertical Hamiltonian flows, it would follow that $\tilde\mu$ has an infinite mass, which contradicts the fact that $\tilde\mu$ is a probability measure. 

%
%
%

Interestingly, as already alluded at the end of Section \ref{sec_mainresults}, the above argument does not work if the Hamiltonian curve $\tilde\gamma$ is horizontal or vertical, i.e., if $h_1(\tilde\gamma)=0$ or $h_2(\tilde\gamma)=0$.
Actually, the action mentioned in Remark \ref{rem_action} is free in the interior of the quadrant but not at the boundary.
Therefore, the fact that a microlocal QL cannot charge any horizontal or vertical Hamiltonian curve, as claimed in Item \ref{itemiv}, has to be proved in another way. This is what we do in Section \ref{sec_thm2_2} with a different proof.
\end{remark}

Before coming to that point, we provide in the next section two additional properties, which are not reported in Theorem \ref{thm2} but may be of interest for other purposes.

\subsection{Additional properties}\label{sec_additional}
Since $\lambda=\lambda^1+\lambda^2\rightarrow+\infty$, either $\lambda^1$ and $\lambda^2$ tend to $+\infty$, or only one of the two sequences tends to $+\infty$ while the other remains bounded.

Let us assume, without loss of generality, that $\lambda^1=\lambda^1_i$ remains constant (i.e., $i$ is fixed) and that $\lambda^2=\lambda^2_j\rightarrow+\infty$ as $j\rightarrow+\infty$. We have $\lambda_j=\lambda^1+\lambda^2_j$. 
In this case, let us prove that
\begin{itemize}
\item $(p_1)_*\mu$ is absolutely continuous with respect to $\rho_1$,
\item $(p_2)_*\mu$ a (finite) convex combination of local QLs of $M_2$,
\end{itemize}
where 
$p_\ell:M=M_1\times M_2\rightarrow M_\ell$ is the canonical projection, for $\ell=1,2$.

Since the sequence of probability measures $\vert\phi_j\vert^2\rho$ converges weakly to $\mu=\pi_*\tilde\mu$, it follows that $(p_\ell)_* \vert\phi_j\vert^2\rho$ converges weakly to the measure $(p_\ell)_*\mu$, for $\ell=1,2$.

Let us first prove that $(p_1)_*\mu$ is absolutely continuous with respect to $\rho_1$.
By the Fubini theorem, for every $a\in C^\infty(M_1)$, we have
$$
\left\langle (p_1)_* \vert\phi_j\vert^2\rho, a \right\rangle = \int_{M_1} a(x_1) \int_{M_2} \vert\phi_j(x_1,x_2)\vert^2\, d\rho_2(x_2)\, d\rho_1(x_1) .
$$
Using \ref{MM}, using the Hilbert basis defined by \eqref{Hilbertbasis}, for every $j$ we have
\begin{equation}\label{phi_j}
\phi_j = \sum_{k=1}^{m_i^1} \ \sum_{l=1}^{m_j^2} a_{i,k,j,l}\, \phi^1_{i,k}\, \phi^2_{j,l}, \qquad \sum_{k=1}^{m_i^1} \ \sum_{l=1}^{m_j^2} \vert a_{i,k,j,l}\vert^2 = 1,
\end{equation}
and thus
$$
\vert \phi_j\vert^2 = \sum_{k,k'=1}^{m_i^1} \ \sum_{l,l'=1}^{m_j^2} a_{i,k,j,l}\,\overline{a_{i,k',j,l'}}\, \phi^1_{i,k}\, \overline{\phi^1_{i,k'}}\, \phi^2_{j,l}\, \overline{\phi^2_{j,l'}} .
$$
When integrating $\vert \phi_j\vert^2$ over $M_2$, the cross product terms $\phi^2_{j,l}\, \overline{\phi^2_{j,l'}}$ vanish when $l\neq l'$ and in the sum it remains only the terms for which $l=l'$. Hence, we obtain
$$
\left\langle (p_1)_* \vert\phi_j\vert^2\rho, a \right\rangle = \sum_{k,k'=1}^{m_i^1} b_{k,k',j} \int_{M_1} a(x_1)\, \phi^1_{i,k} (x_1) \, \overline{\phi^1_{i,k'}(x_1)} \, d\rho_1(x_1) 
$$
where
$$
b_{k,k',j} = \sum_{l=1}^{m_j^2} a_{i,k,j,l}\, \overline{a_{i,k',j,l}}. 
$$
Since there is a finite number integers $k,k'$ in $\{1,\ldots,m^1_i\}$ (recall that $i$ is fixed), and since the sequence $(b_{k,k',j})_{j\in\N}$ is bounded, up to some subsequence we have $b_{k,k',j}\rightarrow \beta_{k,k'}$ as $j\rightarrow+\infty$, for all those integers $k,k'$. Hence
$$
\left\langle (p_1)_* \vert\phi_j\vert^2\rho, a \right\rangle \longrightarrow \sum_{k,k'=1}^{m_i^1} \beta_{k,k'} \int_{M_1} a(x_1)\, \phi^1_{i,k} (x_1) \, \overline{\phi^1_{i,k'}(x_1)} \, d\rho_1(x_1) 
$$
as $j\rightarrow+\infty$, and therefore the measure $(p_1)_*\mu$ is absolutely continuous, of density
$$
\frac{d (p_1)_*\mu}{d\rho_1}(x_1) = \sum_{k,k'=1}^{m_i^1} \beta_{k,k'} \, \phi^1_{i,k} (x_1) \, \overline{\phi^1_{i,k'} (x_1)}  \qquad \forall x_1\in M_1 .
$$

\medskip

Let us then prove that $(p_2)_*\mu$ a (finite) convex combination of QLs of $M_2$. By \eqref{phi_j}, for every $j$ we have
$$
\phi_j = \sum_{k=1}^{m^1_i} \phi^1_{i,k} \psi_{k,j}
\qquad\textrm{where}\qquad
\psi_{k,j} = \sum_{l=1}^{m^2_j} a_{i,k,j,l} \phi^2_{j,l}
$$
(we do not put an index $i$ in $\psi_{k,j}$ because $i$ is fixed) is an eigenfunction of $\triangle_2$, of norm
$$
c_k^j = \Vert \psi_{k,j} \Vert_{L^2(M_2,\rho_2)} = \left(\sum_{l=1}^{m^2_j} \vert a_{i,k,j,l}\vert^2 \right)^{1/2} .
$$
By orthogonality of the $\phi^1_{i,k}$, we have
$$
\int_{M_1} \vert\phi_j(x_1,x_2)\vert^2\, d\rho(x_1) = \sum_{k=1}^{m^1_i} \vert\psi_{k,j}(x_2)\vert^2 ,
$$
hence
\begin{equation}\label{14:06}
(p_2)_* \vert\phi_j\vert^2\,\rho = \left( \int_{M_1} \vert\phi_j(x_1,x_2)\vert^2\, d\rho(x_1) \right) \rho_2 = \sum_{k=1}^{m^1_i} \left(c^j_k\right)^2 \left\vert\frac{\psi_{k,j}}{c^j_k}\right\vert^2 \rho_2
\end{equation}
For every $j$, we define the vector $c^j\in\R^{m^1_i}$ whose coordinates are the nonnegative real numbers $c_k^j$, for $k=1,\ldots,m^1_i$. By \eqref{phi_j}, the Euclidean norm of $c^j$ is equal to $1$. Therefore, by compactness (recall that $i$ is fixed), taking a subsequence if necessary, the sequence $(c^j)_{j\in\N^*}$ converges to some $c\in\R^{m^1_i}$, of coordinates $c_k$, for $k=1,\ldots,m^1_i$, such that $\sum_{k=1}^{m^1_i} c_k^2=1$. 

Besides, since $\frac{\psi_{k,j}}{c^j_k}$ is a normalized eigenfunction of $\triangle_2$, up to some subsequence there exists $\tilde\mu^2_k\in\QL(M_2)$ such that the sequence of probability measures $\Big\vert\frac{\psi_{k,j}}{c^j_k}\Big\vert^2 \rho_2$ converges weakly to $\tilde\mu^2_k$, for $k=1,\ldots,m^1_i$.

Passing to the limit in \eqref{14:06}, we finally obtain that
$(p_2)_*\mu = \sum_{k=1}^{m^1_i} c_k^2 \tilde\mu^2_k$,
i.e., $(p_2)_*\mu$ is a finite convex combination of local QLs of $M_2$. In particular, it is invariant under the Hamiltonian flow of $M_2$.

\subsection{Proof of \ref{itemiv}}\label{sec_thm2_2}
Let $\tilde\gamma$ be a nontrivial Hamiltonian curve on the submanifold $\{h=1\}$ of $T^*M$ (which is identified with an open subset of $S^*M$) and let $\gamma=\pi\circ\tilde\gamma=(\gamma_1,\gamma_2)$ on $M=M_1\times M_2$.
Let $\tilde\mu \in \QL(S^*M)$ and let $\mu=\pi_*\tilde\mu$. Let us prove that, under \ref{MM}, we have $\mu(\gamma(\R))=0$.
As noted in Remark \ref{rem_nonhornonvert}, it suffices to prove this fact for $\gamma$ horizontal or vertical. But the proof that we give hereafter is general.

Since the functions $t\mapsto h_1(\tilde\gamma(t))$ and $t\mapsto h_2(\tilde\gamma(t))$ are constant, at least one of the two constants is positive. In the sequel, without loss of generality we assume that $h_1(\tilde\gamma(t))=\mathrm{Cst}>0$ for every $t\in\R$ (this is the case if $\tilde\gamma$ is horizontal, i.e., $h_2(\tilde\gamma)=0$).

%
%

\begin{lemma}\label{lemT}
In the framework of Theorem \ref{thm2}, or of Remark \ref{rem_boundary} in the Riemannian case with boundary, 
assume that $\gamma([0,T])$ is contained in the interior of $M$ (i.e., does not intersect the boundary) for some $T>0$. Then $\mu(\gamma([0,T]))=0$. 
\end{lemma}

\begin{proof}
For every $\varepsilon>0$, we define the open subset $\omega_\varepsilon$ and the closed subset $F_\varepsilon$ of $M$ by
$$
\omega_\varepsilon =\{x\in M\ \mid\  d_M(x,\gamma([0,T]))<\varepsilon\}, \qquad F_\varepsilon = M\setminus \omega_\varepsilon,
$$
where $d_M$ is an arbitrary distance on $M$ inducing its topology.
Since $h_1(\tilde\gamma(t))=\mathrm{Cst}>0$ and since $h_1$ is homogeneous of degree $d$, we have $\dot\gamma_1=\frac{\partial h_1}{\partial\xi_1}\neq 0$ along $\tilde\gamma$, hence
there exists $\eta>0$ such that, for every $t\in [\eta,T-\eta]$, the curve $\gamma(\cdot)$, restricted to $[t-\eta,t+\eta]$, intersects the vertical set $\{\gamma_1(t)\}\times M_2$ only at the point $\gamma(t)$.

Let $x_1\in M_1$ be arbitrary. We consider the vertical trace of $F_\varepsilon$ above $x_1$, i.e.,
$$
(F_\varepsilon)_{x_1} = \{ x_2\in M_2\ \mid\ (x_1,x_2)\in F_\varepsilon \}.
$$
By the above property satisfied by $\gamma(\cdot)$, there exist (at most) $N=[T/\eta]+1$ points $x_2^k=x_2^k(x_1)$ (depending on $x_1$) in $M_2$, with $k=1,\ldots,N$, such that
\begin{equation}\label{property_neighb}
\{x_1\}\times \left(M_2\setminus \bigcup_{k=1}^N B_2(x_2^k,\varepsilon)\right)\subset (F_\varepsilon)_{x_1}  \qquad \forall \varepsilon\in (0,\eta)
\end{equation}
where $B_2(x_2^k,\varepsilon)$ is the open ball in $M_2$ of center $x_2^k$ and of radius $\varepsilon$ (see Figure \ref{fig1}), for some distance on $M_2$ inducing its topology, and moreover $N$ does not depend on $x_1\in M_1$ (actually, we can take $N=1$ if $T$ is small enough).
\begin{figure}[h]
\begin{center}
\resizebox{8.5cm}{!}{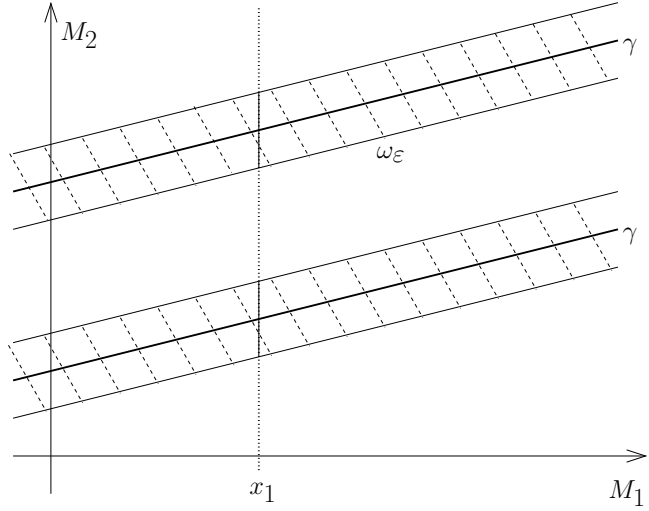}
\end{center}
\caption{Illustration of \eqref{property_neighb}.}\label{fig1}
\end{figure}
Using \ref{MM} and Proposition~\ref{main2}, we have
\begin{equation}\label{13:03}
\g_M(F_\varepsilon)\geq \g_{M_1}\left(x_1\mapsto\g_{M_2}((F_\varepsilon)_{x_1})\right) \geq \g_{M_1}\left( \theta_\varepsilon \right) \qquad \forall \varepsilon\in (0,\eta)
\end{equation}
where
$$
\theta_\varepsilon(x_1)=\g_{M_2}\left( M_2\setminus \bigcup_{k=1}^NB_2(x_2^k,\varepsilon) \right) \qquad \forall x_1\in M_1 .
$$
Note that $\theta_\varepsilon(x_1)\in[0,1]$. We claim that
\begin{equation}\label{m0734}
\lim_{\varepsilon\rightarrow 0} \,\inf_{x_1\in M_1}\theta_\varepsilon(x_1)=1.
\end{equation}
Admitting \eqref{m0734} temporarily, it follows that for every $k\in \N^*$ there exists $\varepsilon_k>0$ such that $\theta_\varepsilon(x_1)\geq 1-1/k$ for every $x_1\in M_1$ and for every $\varepsilon\in[0,\varepsilon_k]$, and thus, using \eqref{13:03}, $\g_M(F_\varepsilon)\rightarrow 1$ as $\varepsilon\rightarrow 0$.
We now conclude like at the end of Section \ref{sec_proof_ii}: 
since $F_\varepsilon$ is closed, we have $\g_M(F_\varepsilon)\leq \g'_M(F_\varepsilon)$ by Lemma \ref{lemggprime} and thus $\g'_M(F_\varepsilon)\rightarrow 1$ as $\varepsilon\rightarrow 0$. Hence, for every $\mu\in\QL(M)$, we have $\mu(F_\varepsilon)\rightarrow 1$ and thus $\mu(\omega_\varepsilon)\rightarrow 0$ as $\varepsilon\rightarrow 0$.
Since $\gamma(\R)=\cap_{\varepsilon>0}\omega_\varepsilon$, we conclude that $\mu(\gamma(\R))=0$.

It remains to prove \eqref{m0734}. By contradiction, if \eqref{m0734} is not true, there exist $\alpha_0\in (0,1)$, a sequence of positive real numbers $\varepsilon_j$ converging to $0$ and a sequence of points $x_1^j\in M_1$ (with $j\in\N$) such that $\theta_{\varepsilon_j}(x_1^j) \leq 1-\alpha_0$ for every $j\in\N$.
Hence, for every $j\in\N$ there exist $N$ points $x_2^{k,j}=x_2^k(x_1^j)\in M_2$, with $k=1,\ldots,N$, and an eigenfunction $\phi_{\varepsilon_j}$ of $\triangle_2$, of norm $1$ in $L^2(M_2,\rho_2)$, such that
$$
\int_{M_2\setminus \bigcup_{k=1}^N B_2(x_2^{k,j},\varepsilon_j)}\vert \phi_{\varepsilon_j}\vert^2\, d\rho_2\leq 1-\frac{\alpha_0}{2} .
$$
By compactness, taking a subsequence if necessary, there exists $x_1\in M_1$ such that $x_1^j$ converges to $x_1$ in $M_1$, and there exist points $x_2^{k,\infty}\in M_2$ such that $x_2^{k,j}$ converges to $x_2^{k,\infty}$ as $j\rightarrow +\infty$, for $k=1,\ldots,N$. Then, there exists $\alpha_1>0$ such that, for every $\varepsilon>0$,
$$
\int_{\bigcup_{k=1}^N \overline{B}_2(x_2^{k,\infty},\varepsilon)}\vert \phi_{\varepsilon_j}\vert^2\, d\rho_2\geq \alpha_1
$$
for every $j$ large enough (so that, at least, $\varepsilon_j<\varepsilon$).
Here, $ \overline{B}_2(x_2^{k,\infty},\varepsilon)$ is the closed geodesic ball of center $x_2^{k,\infty}$ and of radius $\varepsilon$.
Taking another subsequence if necessary, there exists $\mu_2\in \mathcal{Q}(M_2)$ such that the sequence of probability measures $\vert \phi_{\varepsilon_j}\vert^2\, \rho_2$ converges weakly to $\mu_2$ as $j\rightarrow +\infty$. Since $\bigcup_{k=1}^N \overline{B}_2(x_2^{k,\infty},\varepsilon)$ is closed, using the Portmanteau theorem (see Appendix \ref{cintre}), we infer that $\mu_2(\bigcup_{k=1}^N \overline{B}_2(x_2^{k,\infty},\varepsilon)) \geq \alpha_1$. Since this inequality is valid for any $\varepsilon>0$, we obtain that $\sum_{k=1}^N\mu_2 (\{x_2^{k,\infty}\})\geq \alpha_1$. 

Hence there exists a point $x_2\in M_2$ (which is one of the points $x_2^{k,\infty}$) such that $\mu_2(\{x_2\})>0$.
The probability measure $\mu_2\in\QL(M_2)$ is the image under the canonical projection $\pi_2:S^*M_2\rightarrow M_2$ of a microlocal QL $\tilde\mu_2$ which is a probability Radon measure on $S^*M_2$. As a consequence of the Egorov theorem, $\tilde\mu_2$ is invariant under the Hamiltonian flow on $S^*M_2$ (see, e.g., \cite{Zworski}). The propagation property, along with the fact that $\mu(\{x_2\})>0$, implies that $\mu_2([0,T])=+\infty$, which raises a contradiction with the fact that $\mu_2$ is a probability measure. Therefore, \eqref{m0734} is proved.
\end{proof}

To conclude the proof of Item \ref{itemiv}, Lemma \ref{lemT} implies that $\tilde\mu(\tilde\gamma([0,T]))=0$ for every $T>0$ and every Hamiltonian curve $\tilde\gamma$ on $\{h=1\}$, not necessarily periodic. When $\tilde\gamma$ is not periodic this fact is obvious since the probability measure $\tilde\mu$ is invariant under the Hamiltonian flow and has finite mass. Therefore this is only for periodic Hamiltonian curves that the result is interesting, and this is the contents of Item \ref{itemiv}.

\medskip

In the case of Riemannian manifolds with boundary (statement done in Remark \ref{rem_boundary}), the assumption of Dirichlet or Neumann boundary conditions ensures the validity of the invariance property of QLs under the generalized geodesic flow, even when reflecting at the boundary (see \cite{lebeau2} for the Dirichlet case and see \cite[Footnote 3 page 996]{LLTP} for the Neumann case). 
If it happens that $\gamma$ (or a part of it) is entirely contained in the boundary of $M$, for instance if $\gamma_1$ is in $
M_1$ and $\gamma_2$ entirely lies in $\partial M_2$, then Lemma \ref{lemT} is straightforwardly extended.
Finally, in the particular case where the curve $\gamma=\pi\circ\tilde\gamma$ hits \emph{simultaneously} the boundaries of $M_1$ and $M_2$, it reflects along itself and thus uniqueness is not lost, which ensures the well-posedness of the Melrose-Sj\"ostrand bicharacteristic flow (for more details, see \cite[Remark 1.9]{LLTP}).
The conclusion follows.

\appendix

\section{Appendix}

\subsection{Portmanteau theorem} \label{cintre}
Let us recall the so-called Portmanteau theorem (see, e.g., \cite{Billingsley}).
Let $X$ be a metric space, endowed with its Borel $\sigma$-algebra. 
Let $\mu$ and $\mu_n$, $n\in\N^*$, be finite Borel measures on $X$. Then the following items are equivalent:
\begin{itemize}
\item $\mu_n\rightarrow\mu$ for the narrow topology, i.e., $\int_X f\, d\mu_n\rightarrow\int_X f\, d\mu$ for every bounded continuous function $f$ on $X$;
\item $\int_X f\, d\mu_n\rightarrow\int_X f\, d\mu$ for every Borel bounded function $f$ on $X$ such that $\mu(\Delta_f)=0$, where $\Delta_f$ is the set of points at which $f$ is not continuous;
\item $\mu_n(B)\rightarrow\mu(B)$ for every Borel subset $B$ of $X$ such that $\mu(\partial B)=0$;
\item $\mu(F)\geq\limsup\mu_n(F)$ for every closed subset $F$ of $X$, and $\mu_n(X)\rightarrow\mu(X)$;
\item $\mu(O)\leq\liminf\mu_n(O)$ for every open subset $O$ of $X$, and $\mu_n(X)\rightarrow\mu(X)$.
\end{itemize}

\subsection{Some facts on invariant measures in the Riemannian case}\label{app:measures}
Let $(M,g)$ be a smooth compact Riemannian manifold (not necessarily a product), endowed with the normalized canonical measure $\rho$. Let $\triangle$ be the Laplace-Beltrami operator.
Its principal symbol is $h=g^*$, the co-metric on $T^*M$. The co-sphere bundle is canonically identified with the submanifold $\{h=1\}$ (the characteristic manifold is trivial).

The main reason why we restrict ourselves to the Riemannian case in this section is that, in this case, we have the following property for $h$:
\begin{equation}\label{spec_Riema}
\xi_1\neq\xi_2 \Rightarrow \frac{\partial h}{\partial\xi}(x,\xi_1)\neq \frac{\partial h}{\partial\xi}(x,\xi_2) 
\end{equation}
for every $x\in M$ and for all $\xi_1,\xi_2\in T^*_xM$, in local coordinates.

We recall that, given a periodic geodesic $\tilde\gamma$ on $S^*M$, the Dirac measure $\delta_{\tilde\gamma}$ on $S^*M$  is defined by $\delta_{\tilde\gamma}(a)=\frac{1}{T}\int_0^Ta(\tilde\gamma(t))\, dt$, for every $a\in C^0(S^*M)$, where $T$ is the period of $\gamma$.
Setting $\gamma=\pi\circ\tilde\gamma$ where $\pi:S^*M\rightarrow M$ is the canonical projection, the Dirac measure $\delta_\gamma$ on $M$ is defined by $\delta_\gamma(f)=\frac{1}{T}\int_0^Tf(\gamma(t))\, dt$, for every $f\in C^0(M)$. 
Obviously, we have $\pi_*\delta_{\tilde\gamma}=\delta_\gamma$.

Before stating the next result, we recall a useful fact. Let $\Phi:X\rightarrow Y$ be a measurable mapping, with $X$ and $Y$ separable metric spaces. Let $\mu$ be a Radon measure on $X$ and let $\Phi_*\mu$ be its image under $\Phi$. 
We recall that the support of $\mu$ is the closed subset $\supp(\mu)$ of $X$ defined as the set of all $x\in X$ such that $\mu(U)>0$ for any neighborhood $U$ of $x$. If $\Phi$ is continuous and proper then $\Phi(\supp(\mu))=\supp(\Phi_*\mu)$.

Hereafter, we establish a decomposition of invariant probability Radon measures on $S^*M$ with respect to any Dirac measure along a periodic geodesic. The invariance is understood with respect to the geodesic flow.
Note that, by propagation, a finite invariant Radon measure can involve a Dirac part only if this part is a Dirac $\delta_{\tilde\gamma}$ along a geodesic and moreover the geodesic $\tilde\gamma$ must be periodic (due to finiteness of the measure). Denote by $\mathcal{I}(S^*M)$ the set of invariant probability Radon measures on $S^*M$.

\begin{proposition} \label{propQLM}
Let $\tilde\mu \in \mathcal{I}(S^*M)$ and let $\mu=\pi_*\tilde\mu$.
Let $\tilde\gamma$ be a periodic geodesic on $S^*M$ and let $\gamma=\pi\circ\tilde\gamma$. Then:
\begin{itemize}
\item $\tilde\mu=\delta_{\tilde\gamma}$ if and only if $\mu = \delta_\gamma$. 
\item There exists a nonnegative Radon measure $\tilde\mu_1$ on $S^*M$ that is invariant under the geodesic flow, such that $\tilde\mu=a \delta_{\tilde\gamma}+\tilde\mu_1$, with $a=\tilde\mu(\tilde\gamma(\R))$ and $\tilde\mu_1(\tilde\gamma(\R))=0$.
Moreover, setting $\mu_1=\pi_*\tilde\mu_1$, we have $\mu =a\delta_\gamma+ \mu_1$ with $a=\mu(\gamma(\R))$ and $\mu_1(\gamma(\R)) =0$.
\end{itemize}
\end{proposition}

\begin{proof}[Proof of Proposition \ref{propQLM}.]
%
Let us prove the first point.
If $\tilde\mu=\delta_{\tilde\gamma}$, then $\mu=\pi_*\delta_{\tilde\gamma}=\delta_\gamma$. 
Conversely, if $\mu=\pi_*\tilde\mu = \delta_\gamma=\pi_*\delta_{\tilde\gamma}$, then $\supp(\tilde\mu)\subset\pi^{-1}(\gamma(\R))$. Let us prove that we have exactly $\supp(\tilde\mu)=\tilde\gamma(\R)$. 
Since $\tilde\mu$ is invariant under the geodesic flow $\varphi_t$ (meaning that $\varphi_t^*\tilde\mu=\tilde\mu$ for any $t\in\R$), we must have $\varphi_t(\supp(\tilde\mu))=\supp(\tilde\mu)\subset\pi^{-1}(\gamma(\R))$ for any $t\in\R$.
Given any $z\in\supp(\tilde\mu)$, we must have $\pi(z)\in\pi(\supp(\tilde\mu))=\supp(\pi_*\tilde\mu)=\gamma(\R)$, and therefore there exists $s_0\in\R$ such that $\pi(z)=\gamma(s_0)$. Since $\varphi_t(z)\in\pi^{-1}(\gamma(\R))$ for any $t\in\R$, we must have $\pi(\varphi_t(z))=\gamma(s(t))$ for some $s(t)$, with $s(0)=s_0$. The function $t\mapsto s(t)$ must be strictly monotone around $s_0$, and since $\gamma$ cannot have two distinct lifts in $S^*M$, it follows that $z\in\tilde\gamma$ and $\varphi_t(z)=\tilde\gamma(s_0+t)$. The first point is proved.

Let us prove the second point.
We denote by $\vec{h}$ the geodesic Hamiltonian field, i.e., the Hamiltonian vector field on $T^*M$ associated with the Hamiltonian $h$.

\begin{lemma}\label{lem_geomR}
Given any geodesic $\gamma=\pi\circ\tilde\gamma$ and any $T>0$, the vector field $\vec{h}$ is transverse to $\pi^{-1}(\gamma((0,T)))\setminus\tilde\gamma((0,T))$.
\end{lemma}

Note that $\pi^{-1}(\gamma((0,T)))$ is a fiber bundle with base the one-dimensional manifold $\gamma((0,T))$.

\begin{proof}[Proof of Lemma \ref{lem_geomR}.]
Taking an arbitrary point $x=\pi(x,\xi_1)\in\gamma((0,T))$, the tangent space to $\pi^{-1}(\gamma((0,T)))$ at any point $(x,\xi)$ of the vertical fiber above $x$ does not depend on $\xi$ and is spanned by all $(\frac{\partial h}{\partial\xi}(x,\xi_1),*)$. 
Thanks to \eqref{spec_Riema}, $\vec{h}$ is transverse to $\pi^{-1}(\gamma((0,T)))$ at $(x,\xi_2)$ (see Figure \ref{fig_geod}). 
\end{proof}

\begin{figure}[h]
\centerline{\scalebox{0.37}{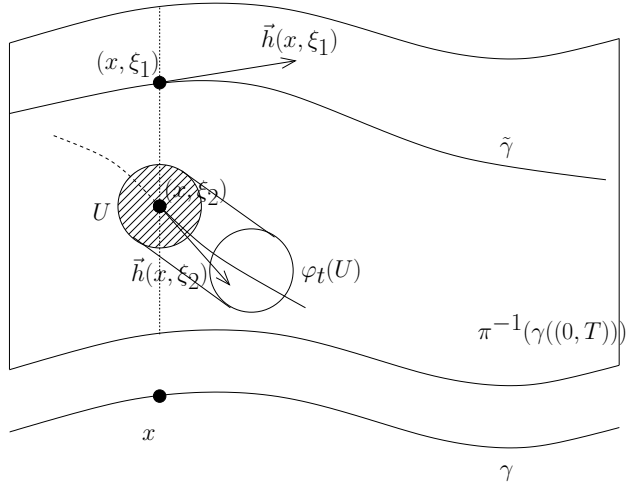}}
\caption{Illustration of Lemmas \ref{lem_geomR} and \ref{lemraynul}.}
\label{fig_geod}
\end{figure}

From Lemma \ref{lem_geomR}, we deduce another general lemma.

\begin{lemma}\label{lemraynul}
Let $\tilde\mu$ be a nonnegative finite invariant Radon measure on $S^*M$. Given any geodesic $\gamma=\pi\circ\tilde\gamma$ and any $T>0$, we have $\tilde\mu(\pi^{-1}(\gamma((0,T)))\setminus\tilde\gamma((0,T)))=0$.
\end{lemma}

\begin{proof}[Proof of Lemma \ref{lemraynul}.]
We argue by contradiction. If $\tilde\mu(\pi^{-1}(\gamma((0,T)))\setminus\tilde\gamma((0,T)))>0$, then there exist $z\in\pi^{-1}(\gamma((0,T)))\setminus\tilde\gamma((0,T))$ and a neighborhood $U$ of $z$ in the manifold $\pi^{-1}(\gamma(\R))$ such that $\tilde\mu(U)>0$. Let us propagate $U$ under the geodesic flow $\varphi_t$. By Lemma \ref{lem_geomR}, the vector field $\vec {h}$ is transverse to $\pi^{-1}(\gamma((0,T)))$. Hence if $U$ and $t>0$ are sufficiently small then $\varphi_t(U) \cap \pi^{-1}(\gamma((0,T))) = \emptyset$, and actually the union of all $\varphi_s(U)$ with $0\leq s\leq t$ is a cylinder (denoted by $\mathcal{C}$) with distinct layers $\varphi_s(U)$ (see Figure \ref{fig_geod}). 

Now, since $\tilde\mu$ is invariant under the geodesic flow, we have $\tilde\mu(\varphi_s(U))=\tilde\mu(U)>0$. It follows that $\tilde\mu(\mathcal{C})=+\infty$, which contradicts the fact that $\tilde\mu$ is finite.
\end{proof}

We are now in a position to prove the second point. We set $a=\tilde\mu(\tilde\gamma(\R))$ and $\tilde\mu_1=\tilde\mu-a\delta_{\tilde\gamma}$. Given any measurable subset $B$ of $S^*M$, by definition of $\tilde\mu_1$ we have $\tilde\mu_1(B)=\tilde\mu_1(B\setminus\tilde\gamma(\R))+\tilde\mu_1(\tilde\gamma(\R))=\tilde\mu(B\setminus\tilde\gamma(\R))$, from which we infer that $\tilde\mu_1$ is a nonnegative measure that is invariant under the geodesic flow.
Note that $\tilde\mu_1(\tilde\gamma(\R))=0$.

Let us prove that $(\pi_*\tilde\mu_1)(\gamma(\R))=0$. This follows from Lemma \ref{lemraynul}, by writing that $\pi_*\tilde\mu_1=\pi_*\tilde\mu-a\delta_\gamma$ (because $\pi_*\delta_{\tilde\gamma} = \delta_\gamma$) and
$$
(\pi_*\tilde\mu_1)(\gamma(\R)) = (\pi_*\tilde\mu)(\gamma(\R)) - a
= \tilde\mu(\pi^{-1}(\gamma(\R))) - \tilde\mu(\tilde\gamma(\R))
= \tilde\mu(\pi^{-1}(\gamma(\R))\setminus\tilde\gamma(\R))=0.
$$
The proposition is proved.
\end{proof}

\subsection{Proof of Proposition~\ref{dilatation}}\label{proof:propdilat}
Let $A$ be the set of all possible $s>0$ such that the spectrum of $\triangle_1 + \triangle_{2,s}$ does not satisfy \ref{MM}.
If $A$ is empty, there is nothing to do. Then, let us assume that $A$ is nonempty. 

Given any $s\in A$, using that \ref{MM} is equivalent to \eqref{simplicity2}, there exist two distinct pairs $(i,j),(i',j')\in (\N^*)^2$ such that $\lambda^1_i+s^\alpha \lambda^2_j=\lambda^1_{i'}+s^\alpha \lambda^2_{j'}$, hence necessarily $\lambda^1_i\neq\lambda^1_{i'}$ and $\lambda^2_j\neq\lambda^2_{j'}$, and we set $F(s)=(\lambda^1_i,\lambda^2_j, \lambda^1_{i'},\lambda^2_{j'})$. By the axiom of choice, we have thus constructed a map $F:A\rightarrow \R^4$, whose range is at most countable.

We claim that $F$ is injective. Indeed, assume that $F(s)=(\lambda^1_i,\lambda^2_j, \lambda^1_{i'},\lambda^2_{j'})=F(s')$ with $(s,s')\in A^2$. Then $\lambda^1_i+s^\alpha \lambda^2_j=\lambda^1_{i'}+s^\alpha \lambda^2_{j'}$ and $\lambda^1_i+(s')^\alpha \lambda^2_j=\lambda^1_{i'}+(s')^\alpha \lambda^2_{j'}$. Hence $\lambda^1_i-\lambda^1_{i'} = s^\alpha(\lambda^2_{j'}-\lambda^2_j) = (s')^\alpha(\lambda^2_{j'}-\lambda^2_j)$, and since $\lambda^2_j\neq \lambda^2_{j'}$ we infer that $s=s'$.

Since the range of $F$ is at most countable, we conclude that $A$ is at most countable.

\subsection{Proof of Lemma~\ref{lem_square}}\label{proof:lem_square}
In \cite{jaffard}, it is shown how to apply three results of \cite{Kahane} to the sequence of points of $\R^3$ given by $\Lambda = \left\{ (j,k,j^2+k^2)\ \mid\ j,k\in\Z \right\}$ in order to prove that, given any open subset $\omega$ of $(0,\pi)^2$, given any $T>0$, there exists $C_T(\omega)>0$ such that
\begin{equation}\label{ineg_jaffard}
\int_0^T\int_\omega \bigg\vert \sum_{j,k\in\Z} a_{jk} \sin(jx)\sin(ky)e^{i(j^2+k^2)t}\bigg\vert^2 \, dx\, dy\, dt   \geq   C_T(\omega)  \sum_{j,k\in\Z} \vert a_{jk}\vert^2
\end{equation}
for every $(a_{jk})_{j,k\in\Z}\in \ell^2(\C)$.
The consequence of \eqref{ineg_jaffard} that is given in \cite{jaffard} is the observability (and thus the controllability) property for plate and Schr\"odinger equations in the square, for any open subset $\omega$ and for any $T>0$ (without geometric control condition).

Here, we give another consequence of \eqref{ineg_jaffard}.
Given any integer $\lambda\in(\N^*)^2+(\N^*)^2$ that is the sum of two squares of positive integers, we define the set
$$
E_\lambda = \{(j,k)\in(\N^*)^2 \ \mid\ \lambda=j^2+k^2 \} 
$$
Then, as a consequence of \eqref{ineg_jaffard}, we have
$$
\int_\omega \bigg\vert \sum_{(j,k)\in E_\lambda} a_{jk} \sin(jx)\sin(ky)\bigg\vert^2 \, dx\, dy   \geq   \frac{C_T(\omega)}{T}  \sum_{(j,k)\in E_\lambda} \vert a_{jk}\vert^2
$$
for every $(a_{jk})_{(j,k)\in E_\lambda}\in \ell^2(\C)$.
Then, \eqref{conj_grebenkov} follows.

\paragraph{Acknowledgment.} 
We warmly thank Yves Colin de Verdi\`ere for many fruitful discussions.
The first author is supported by the project THESPEGE (APR IA), R\'egion Centre-Val de Loire, France, 2018-2020.

\end{document}